\documentclass[12pt]{amsart}

\usepackage[linesnumbered,ruled,vlined]{algorithm2e}
\SetKwInput{KwInput}{Input}                
\SetKwInput{KwOutput}{Output}              
\SetAlgorithmName{Procedure}{procedure}{List of procedures}

\usepackage{amssymb}
\usepackage{amsthm}
\usepackage{amsmath}
\usepackage{graphicx}
\usepackage{hyperref}
\usepackage{placeins}
\usepackage{tikz-cd}
\newcommand{\abs}[1]{\ensuremath{\left\lvert #1 \right\rvert}}
\newtheorem{thm}{Theorem}
\newtheorem{lm}{Lemma}
\newtheorem{prop}{Proposition}
\newtheorem{dfn}{Definition}
\newtheorem*{conj}{Conjecture}
\newtheorem{claim}{Claim}
\newtheorem{cor}{Corollary}
\newtheorem*{mainres}{Main result}
\newtheorem*{thm*}{Theorem}
\theoremstyle{definition}
\newtheorem{remark}{Remark}

\author{Ognjen To\v{s}i\'{c} }\address{Mathematical Institute, University of Oxford}
\title{Confirming Brennan's conjecture numerically on a counterexample to Thurston's $K=2$ conjecture}
\date{}
\begin{document}
\begin{abstract}It was shown by Bishop that if Thurston's $K = 2$ conjecture holds for some planar domain, then Brennan's conjecture holds for the Riemann map of that domain as well. In this paper we show numerically that the original counterexample to Thurston's $K=2$ conjecture given by Epstein, Marden and Markovi\'{c} is not a counterexample to Brennan's conjecture. 
\end{abstract}
\maketitle
    \section{Introduction}
    \subsection{Thurston's $K=2$ conjecture}
        Let $\Omega$ be a simply connected proper subdomain of $\mathbb{C}\subset \mathbb{S}^2=\partial\mathbb{H}^3$ where $\mathbb{H}^3$ is the 3-dimensional hyperbolic space. Let $\tilde{\Omega}$ be the union of all hyperbolic half-spaces $H$ such that $H\cap \partial\mathbb{H}^3\subseteq\Omega$. Then we define the dome of $\Omega$ to be $\mathrm{Dome}(\Omega)=\partial\tilde{\Omega}\cap \mathbb{H}^3$. It is known (see the book by Epstein and Marden \cite{em-convex-hulls-sullivan} for a detailed account) that $\mathrm{Dome}(\Omega)$ with the path metric induced from $\mathbb{H}^3$ is isometric to the unit disk $\mathbb{D}$ with its hyperbolic metric. Using this isometry, we give $\mathrm{Dome}(\Omega)$ a conformal structure.
        \par Note that $\Omega$ and $\mathrm{Dome}(\Omega)$ share the common boundary $\partial\Omega$, so we set 
        \begin{align*}
            \mathcal{F}=\left\{f:\Omega\to\mathrm{Dome}(\Omega):f(x)=x\text{ for all }x\in\partial\Omega\text{ and }f\text{ is continuous}\right\}.
        \end{align*}
        Let $\text{M\"ob}(\Omega)$ be the set of all M\"obius transformations that preserve $\Omega$. Each such transformation extends to an isometry of $\mathbb{H}^3$ preserving $\mathrm{Dome}(\Omega)$, so we define
        \begin{align*}
            \mathcal{F}_\text{eq}=\{f\in\mathcal{F}:f\circ\gamma=\gamma\circ f\text{ for all }\gamma\in\text{M\"ob}(\Omega)\}.
        \end{align*}
    \begin{dfn}
        Given a simply connected proper subdomain $\Omega$ of the plane, we define 
        \begin{align*}
            K(\Omega)&=\inf\{K>0:\text{there exists a }K\text{-quasiconformal map }f\in\mathcal{F}\}, \text{ and }\\
            K_\text{eq}(\Omega)&=\inf\{K>0:\text{there exists a }K\text{-quasiconformal map }f\in\mathcal{F}_\text{eq}\}.
        \end{align*}
    \end{dfn}
    It was shown by Sullivan that $\sup_\Omega K(\Omega)<\infty$ in \cite{sullivan-original}. A more detailed proof of this result and a proof that $\sup_\Omega K_\mathrm{eq}(\Omega)<\infty$ was given by Epstein and Marden in \cite{em-convex-hulls-sullivan}. Thurston conjectured the following result. 
    \begin{conj}[Thurston's $K=2$ conjecture]
        For any simply connected proper subdomain $\Omega\subset\mathbb{C}$, we have $\sup_\Omega K(\Omega)=\sup_\Omega K_\mathrm{eq}(\Omega)=2$. 
    \end{conj}
    Epstein, Marden and Markovi\'{c} constructed counterexamples for this conjecture in \cite{emm-2} for the equivariant case and in \cite{emm-log-spiral} for the general case. Komori and Matthews then constructed in \cite{komori} a more explicit counterexample for the equivariant case building on the ideas of \cite{emm-2}.
    \subsection{Brennan's conjecture}
    We remind the reader that given a simply connected proper subdomain $\Omega\subset\mathbb{C}$, by the Riemann mapping theorem there exists a conformal map $F:\Omega\to\mathbb{D}$. Brennan conjectured the following regarding the growth of $F'$ near $\partial\Omega$. 
    \begin{conj}[Brennan's conjecture]
        Let $\Omega$ be a simply connected proper subdomain of $\mathbb{C}$, and let $F:\Omega\to\mathbb{D}$ be a conformal map. Then $F'\in L^{p}(\Omega)$ for all $\frac{4}{3}<p<4$.
    \end{conj}
    Brennan's conjecture in full generality is still open.
    \subsection{Bishop's theorem}
    Bishop showed in \cite{bishop-1} that if Thurston's $K=2$ conjecture holds for some domain, then Brennan's conjecture holds for that domain as well.
    \begin{thm}\label{thm:bishop} Let $\Omega$ be a simply connected proper subdomain of the complex plane, and let $F:\Omega\to\mathbb{D}$ be a conformal map. Then $F'\in L^p(\Omega)$ for all $p<\frac{2 K(\Omega)}{K(\Omega)-1}$. In particular if $K(\Omega)=2$, then Brennan's conjecture holds for $\Omega$. 
    \end{thm}
    It was shown by Epstein, Marden and Markovi\'{c} in \cite{emm-5} that Thurston's $K=2$ conjecture holds for domains in $\mathbb{C}$ that are convex (in the Euclidean sense), by Theorem \ref{thm:bishop} therefore proving Brennan's conjecture in this case. Bishop showed in \cite{bishop-sullivan} that $K(\Omega)\leq 7.82$ for all simply connected planar domains $\Omega$, therefore giving a new proof that $F'\in L^{2.29}(\Omega)$ for conformal maps $F:\Omega\to\mathbb{D}$.
    \subsection{Our results}\label{subsec:results}
    By Theorem \ref{thm:bishop}, any counterexample to Brennan's conjecture would also be a counterexample to Thurston's $K=2$ conjecture. In this paper we investigate if Brennan's conjecture holds for a particular counterexample to Thurston's $K=2$ conjecture.
    \par In this section we denote the counterexample to Thurston's $K=2$ conjecture from \cite{komori} by $\Omega$. In \cite{komori} it was shown that $K(\Omega)=K_\mathrm{eq}(\Omega)>2$. We show numerically that Brennan's conjecture holds for $\Omega$. 
    \begin{mainres}
        Let $F:\Omega\to\mathbb{D}$ be a conformal map and set $p_\star=\sup\{p>0:F'\in L^p(\Omega)\}$. We show numerically that $5.52<p_\star<5.54$. In particular $p_\star>4$ so Brennan's conjecture holds for $\Omega$.
    \end{mainres}
    The domain $\Omega$ is a connected component of the domain of discontinuity of an explicit Kleinian once-punctured torus group. Our results strongly suggest that Brennan's conjecture holds for all domains constructed in this way. Equivalently, our results suggest that Brennan's conjecture holds for all quasidisks invariant under a group of M\"obius transformations, such that the quotient is a once-punctured torus. 
    \subsection{General strategy}\label{subsec:gen-strat}
    The domain $\Omega$ has an action by a Kleinian once-punctured torus group $\Gamma$. Write $F:\Omega\to\mathbb{D}$ for its Riemann mapping and $f=F^{-1}$ for its inverse. Let $\Gamma_0=F\circ\Gamma\circ F^{-1}$ be the once-punctured torus Fuchsian group obtained by conjugating $\Gamma$ by $F$, and let $\rho:\Gamma_0\to\Gamma$ be the induced isomorphism. Let $\Phi$ be a fundamental domain for the action of $\Gamma_0$ on $\mathbb{D}$. Then $f(\Phi)$ is a fundamental domain for the action of $\Gamma$ on $\Omega$.
    \par Since $\Omega$ and $\Gamma$ are known, we can numerically compute the Riemann map $F:\Omega\to\mathbb{D}$ using Schwarz-Christoffel mappings and a polygonal approximation to $\Omega$. The quality of this estimate deteriorates near the boundary $\partial\Omega$, making it imprecise to check if $F'\in L^p(\Omega)$ directly. However for $\gamma\in\Gamma$, the behavior of $F'$ on $\gamma f(\Phi)$ is controlled by $\gamma$ and $\rho^{-1}(\gamma)$. Therefore $F'\in L^p(\Omega)$ if and only if a certain series over $\Gamma_0$ depending on $p, \Gamma_0$ and $\rho$ converges.
    \par We note that the estimates of $F$ have higher accuracy away from $\partial\Omega$. We can hence still use them to reliably estimate $\rho$ and $\Gamma_0$. Then we use these estimates to check numerically if the series mentioned in the previous paragraph converges.
    \subsection{A more detailed outline of the argument and the computation}\label{subsec:outline}
     The paper is divided into a theoretical section \S\ref{section:theory}, a brief section where we describe $\Omega$ and $\Gamma$ in more detail \S\ref{section:grafting}, and a numerical methods and results section \S\ref{section:methods}.
    \par In the theory part, we show that $\int_\Omega\abs{F'}^p dxdy=\int_{\mathbb{D}}\abs{f'}^{2-p}dxdy$ is equal to a certain infinite series depending on $\Gamma_0$ and $\rho$, up to a bounded multiplicative error. We do this by decomposing $\mathbb{D}$ into $\gamma\Phi$, where $\Phi$ is a fundamental domain for the action of $\Gamma_0$ on $\mathbb{D}$. We express the integral $\int_{\gamma\Phi}\abs{f'}^{2-p}dxdy$ in terms of $\gamma$ and $\rho(\gamma)$, up to a bounded multiplicative error. We work in greater generality, considering Riemann maps $f:\mathbb{D}\to\Omega$ that conjugate a Fuchsian group $\Gamma$ to a Kleinian group. We do this in \S\ref{section:theory}, where the main result is Theorem \ref{thm:main-estimate}. 
    \par In the proof of Theorem \ref{thm:main-estimate}, issues arise since $\Phi$ is not assumed to be compact. This is handled by showing that on a horoball $H$, the derivative $\abs{f'}$ achieves its minimum at the closest point of $H$ to the origin, up to a bounded multiplicative error. This is Lemma \ref{lm:f'-estimate} and is the most involved part of the proof of Theorem \ref{thm:main-estimate}.
    \par We now briefly describe how $\Omega$ is constructed in \cite{komori}. They start with a hyperplane in $\mathbb{H}^3$ along with a Fuchsian once-punctured torus group that preserves it. Fix a hyperbolic element in this group, and consider the orbit of its axis. This is a discrete set of geodesics, along which they bend the hyperplane. The resulting pleated plane intersects the boundary of $\mathbb{H}^3$ in a curve that bounds $\Omega$. We give an equivalent form of their construction in \S\ref{section:grafting} that does the bending in the complex plane, without mentioning $\mathbb{H}^3$.
    \par This construction makes it easy to identify a point on the boundary of $\Omega$, and to see that the orbit of any point on $\partial\Omega$ is dense in $\partial\Omega$. We use this observation to construct finite polygons that approximate $\Omega$. These approximations are described in \S\ref{subsec:polygonal}.
    \par We use Schwarz-Christoffel mappings to numerically compute an approximation to the Riemann map $F:\Omega\to\mathbb{D}$. This approximation is used to (numerically) compute the generators of $\Gamma_0$. This is explained in \S\ref{subsec:est-conj-gen}.
    \par We use estimates of $\Gamma_0$ and $\rho$ to compute initial terms of the sum from Theorem \ref{thm:main-estimate}. The terms appear to decay exponentially for $p<5.52$ and to increase exponentially for $p>5.54$. From this we conclude the Main result. This is done in \S\ref{subsec:vals}.
    \subsection*{Acknowledgments}
    I would like to thank Vladimir Markovi\'c for introducing me to this problem and for his advice while working on it, and in particular for his help on Lemma \ref{lm:f'-estimate}. I would also like to thank the reviewers for pointing out an error in an earlier draft of this paper, and for their helpful comments and suggestions. 
    \section{Theoretical results}\label{section:theory}
    Throughout this section, we let $\Gamma$ be a Fuchsian group such that $\mathbb{D}/\Gamma$ is a finite area hyperbolic surface. Let $f:\mathbb{D}\to\Omega$ be a conformal map, normalized so that $f(0)=0$ and $f'(0)=1$. Suppose that $f$ conjugates $\Gamma$ to a Kleinian group, and let $\rho:\Gamma\to\mathrm{PSL}(2,\mathbb{C})$ be the induced homomorphism. Our main theoretical result is the following estimate. 
    \begin{thm}\label{thm:main-estimate}
        Given $q>0$, there exists a constant $C=C(\Gamma, q)$ that depends only on $\Gamma$ and $q$ such that 
        \begin{align*}
            \frac{1}{C}\sum_{\gamma\in\Gamma}\frac{\abs{\gamma'(0)}^{q+2}}{\abs{\rho(\gamma)'(0)}^{q}}\leq \int_\mathbb{D} \frac{1}{\abs{f'(z)}^q}dxdy\leq C\sum_{\gamma\in\Gamma} \frac{\abs{\gamma'(0)}^{q+2}}{\abs{\rho(\gamma)'(0)}^q}.
        \end{align*}
    \end{thm}
    We now outline the proof of Theorem \ref{thm:main-estimate}. The idea is to estimate the integral separately over $\mathcal{H}$ and $\mathbb{D}\setminus\mathcal{H}$, where $\mathcal{H}$ is the union of a certain $\Gamma$-invariant collection of horoballs based at fixed points of the parabolics in $\Gamma$. The proof consists of three steps.
    \begin{enumerate}
        \item For any horoball $H$ in $\mathcal{H}$, we let $z_0$ be the closest point of $H$ to the origin. We show the inequality $\abs{f'(z)}\geq C\abs{f'(z_0)}$ for $z\in H$, for some universal constant $C$. We will later use this to show that the integral from Theorem \ref{thm:main-estimate} over $\mathcal{H}$ is dominated by the integral over $\mathbb{D}\setminus\mathcal{H}$. \par We show $\abs{f'(z)}\geq C\abs{f'(z_0)}$ using the fact that $f$ conjugates a parabolic element preserving $H$ to a parabolic M\"obius transformation. The core of the proof of this bound are estimates on the growth and derivative of a univalent map $g:\mathbb{H}\to\mathbb{C}$ satisfying $g(z+1)=g(z)+1$.
        \par We show the necessary bounds on $g$ as Proposition \ref{prop:koebe-parabolic} in \S\ref{subsection:generic-univalent}. Then we derive $\abs{f'(z)}\geq C\abs{f'(z_0)}$ in Lemma \ref{lm:f'-estimate} in \S\ref{subseciton:derivative-bound}.
        \par In Corollary \ref{cor:estimate-horoball-integral}, \S\ref{subseciton:derivative-bound}, using $\abs{f'(z)}\geq C\abs{f'(z_0)}$ we bound the integral of $\int_H\abs{f'(z)}^{-q}dxdy$ in terms of $\abs{f'(z_0)}$ and $1-\abs{z_0}$.
        \item In \S\ref{subsection:division} we define $\mathcal{H}$ and construct a compact fundamental domain $\Phi^*$ for the action of $\Gamma$ on $\mathbb{D}\setminus\mathcal{H}$. We show an estimate on $\int_{\gamma\Phi^*} \abs{f'(z)}^{-q}dxdy$ as Proposition \ref{prop:compact-piece} in \S\ref{subsection:compact-piece}, making essential use of the compactness of $\Phi^*$. This estimate is in terms of the Euclidean distance of $\gamma\Phi^*$ to the boundary $\partial\mathbb{D}$ of the unit disk $\mathbb{D}$, and the absolute value of the derivative $\abs{f'(\gamma(0))}$.
        \item By equivariance of $f$, the derivative $f'(\gamma(z))$ is related to $\rho(\gamma)'(f(z))$ and $\gamma'(z)$. This allows us to replace $\abs{f'}$ from the estimates in the first two steps with the derivatives $\abs{\rho(\gamma)'}$ and $\abs{\gamma'}$. For hyperbolic isometries $\gamma$ of the disk, we relate $\abs{\gamma(0)}$ and $\abs{\gamma'(0)}$, so we also replace $\abs{\gamma(0)}$ with $\abs{\gamma'(0)}$ in the estimates in the first two steps. This is done in \S\ref{subsection:pf-of-thm}, and concludes the proof of Theorem \ref{thm:main-estimate}.
    \end{enumerate}
    \subsection*{Notation and conventions}
    We write $X\lesssim_{Z,W,...} Y$ if there exists a constant $C=C(Z,W,...)>0$ that depends only on the variables in the subscript, so that $X\leq CY$. We analogously write $X\gtrsim_{Z,W,...}Y$ if $Y\lesssim_{Z,W,...} X$ and $X\approx_{Z,W,...}Y$ if $X\lesssim_{Z,W,...} Y$ and $Y\lesssim_{Z,W,...} X$.
    \par We will use hyperbolic geometry in several places in this paper. We always use the metric of constant curvature $-1$. In particular on the unit disk $\mathbb{D}$ we have the metric $\frac{4\abs{dz}^2}{\left(1-\abs{z}^2\right)^2}$, and on the upper half-plane $\mathbb{H}$ we have the metric $\frac{\abs{dz}^2}{\mathrm{Im}(z)^2}$. Whenever we write $\mathrm{dist}(\cdot, \cdot)$, we are referring to the distance coming from the hyperbolic metric on either $\mathbb{D}$ or $\mathbb{H}$.
    \subsection{Univalent maps and parabolic isometries}\label{subsection:generic-univalent}
    The main result of this subsection is concerned with the growth of univalent maps $g$ defined on the upper half-plane $\mathbb{H}$ that commute with the parabolic $\gamma$ given by $\gamma(z)=z+1$. By taking quotients $\mathbb{H}/\langle\gamma\rangle$ and $\mathbb{C}/\langle\gamma\rangle$, $g$ descends to a univalent map $h:\mathbb{D}\setminus\{0\}\to\mathbb{C}$. The estimates we show on $g$ come from the inequalities on $h$ and its derivatives near $0$, and follow from the general theory of univalent maps. 
    \par In \S\ref{subsub:general-univalent} we recall some general results from the theory of univalent maps. The main result there is Claim \ref{claim:de-branges}. Then in \S\ref{subsub:prop-koebe-parabolic} we state and prove the main result of this subsection.
    \subsubsection{General results on univalent maps}\label{subsub:general-univalent}
    We recall some general theorems about univalent maps, that we later use in \S\ref{subsub:prop-koebe-parabolic}. We then show an estimate on how closely a univalent map $h$ follows its linear approximation $h(0)+h'(0)z$ near $0$. This is Claim \ref{claim:de-branges}.
    \begin{thm*}[Koebe quarter theorem]
        Let $h:\mathbb{D}\to\mathbb{C}$ be a univalent function with $h(0)=0$ and $h'(0)=1$. Then $h(\mathbb{D})$ contains the disk of radius $\frac{1}{4}$ around $0$.
    \end{thm*}
    \begin{thm*}[Koebe distortion theorem]
        Let $h:\mathbb{D}\to\mathbb{C}$ be a univalent function with $h(0)=0$ and $h'(0)=1$. Then 
        \begin{gather*}
            \frac{\abs{z}}{(1+\abs{z})^2}\leq \abs{h(z)}\leq \frac{\abs{z}}{(1-\abs{z})^2},\\
            \frac{1-\abs{z}}{(1+\abs{z})^3}\leq\abs{h'(z)}\leq \frac{1+\abs{z}}{(1-\abs{z})^3},\\
            \frac{1-\abs{z}}{1+\abs{z}}\leq\abs{z\frac{h'(z)}{h(z)}}\leq \frac{1+\abs{z}}{1-\abs{z}}.
        \end{gather*}
        Moreover, the second inequality implies that for any univalent function $f:\mathbb{D}\to\mathbb{C}$, we have for $z,w\in\mathbb{D}$,
        \begin{align*}
            \frac{\abs{f'(z)}(1-\abs{z}^2)}{\abs{f'(w)}(1-\abs{w}^2)}\leq e^{2\mathrm{dist}(z,w)}.
        \end{align*}
    \end{thm*}
    The following Claim is the estimate we use to derive Proposition \ref{prop:koebe-parabolic}, the main result of this subsection.
    \begin{claim}\label{claim:de-branges}
        Let $h:\mathbb{D}\to\mathbb{C}$ be a univalent map. Then whenever $\abs{z}<\frac{1}{2}$, we have
        \begin{align*}
            \abs{h(z)-h(0)-h'(0)z}<100\abs{h'(0)}\abs{z}^2.
        \end{align*}
    \end{claim}
    \begin{remark}\label{rmrk:de-branges-claim}
        The optimal constant in Claim \ref{claim:de-branges} can easily be computed from de Branges' theorem \cite{bieberbach}. However since this is a fairly simple result, we choose to include a more elementary proof (with a non-optimal constant) below.
    \end{remark}
    \begin{proof}[Proof of Claim \ref{claim:de-branges}]
        Without loss of generality we can suppose $h(0)=0$ and $h'(0)=1$. Write 
        \begin{align*}
            h(z)=z+\sum_{n=2}^\infty a_nz^n.
        \end{align*}
        Let $\gamma$ be the circle of radius $\frac{2}{3}$ centered at the origin. Then
        \begin{align*}
            \abs{a_n}=\abs{\frac{1}{2\pi i}\int_{\gamma} \frac{h(z)}{z^{n+1}}dz}\leq\frac{1}{2\pi}\left(\frac{3}{2}\right)^n\int_0^{2\pi} \abs{h\left(\frac{2}{3}e^{i\theta}\right)}d\theta.
        \end{align*}
        By the Koebe distortion theorem, we have 
        \begin{align*}
            \abs{h\left(\frac{2}{3}e^{i\theta}\right)}\leq 6.
        \end{align*}
        Therefore $\abs{a_n}\leq 6\left(\frac{3}{2}\right)^n$. Therefore for $\abs{z}<\frac{1}{2}$, we have 
        \begin{align*}
            \abs{h(z)-z}\leq 6\sum_{n=2}^\infty \left(\frac{3}{2}\abs{z}\right)^n<6\cdot\frac{9}{4}\abs{z}^2\sum_{n=0}^\infty \left(\frac{3}{4}\right)^n<100\abs{z}^2.
        \end{align*}
    \end{proof}
    \subsubsection{Growth of univalent maps that commute with a parabolic M\"obius transformation}\label{subsub:prop-koebe-parabolic}
    We now state and prove the main result of this subsection, Proposition \ref{prop:koebe-parabolic}. As explained at the start of this subsection, starting with a map $g:\mathbb{H}\to\mathbb{C}$ that commutes with $\gamma(z)=z+1$, we construct a map $h:\mathbb{D}\setminus\{0\}\cong\mathbb{H}/\langle\gamma\rangle\to \mathbb{C}/\langle\gamma\rangle\cong\mathbb{C}\setminus\{0\}$. We show that it extends to $0$ with $h(0)=0$ and then apply Claim \ref{claim:de-branges} to $h$. This shows how $h$ behaves near $0$, and hence how $g$ behaves near infinity. We also obtain some information on $g$ from the Koebe distortion theorem and the Koebe quarter theorem applied to $h$.
    \begin{prop}\label{prop:koebe-parabolic}
        There exists a universal constant $C>0$ such that the following holds. Let $g:\mathbb{H}\to\mathbb{C}$ be a univalent holomorphic map with $g(z+1)=g(z)+1$. Then there exists a complex number $\alpha\in\mathbb{C}$ (depending on $g$) so that the following holds.
        \begin{enumerate}
            \item\label{item:g-est} Whenever $\mathrm{Im}(z)\geq C$, we have $\abs{ g(z)-z+\alpha}\leq Ce^{-2\pi\mathrm{Im}(z)}$.
            \item\label{item:g'-est} Whenever $\mathrm{Im}(z)\geq C$, we have $\frac{1}{C}\leq\abs{g'(z)}\leq C$.
            \item\label{item:koebe-quarter} When $\mathrm{Im}(z)>\frac{\log 2}{\pi}-\mathrm{Im}(\alpha)$, we have $z\in g(\mathbb{H})$.
        \end{enumerate}
    \end{prop}
    \begin{proof}
        Note that $\exp(2\pi ig(z))$ is $1$-periodic, so there exists a holomorphic map $h:\mathbb{D}\setminus\{0\}\to\mathbb{C}$ defined by $h(\exp(2\pi iz))=\exp(2\pi ig(z))$. Since $g$ is univalent, so is $h$. We first show that $h$ extends to $0$.
        \begin{claim}
            The map $h$ has a removable singularity at $0$ and can be extended holomorphically so that $h(0)=0$.
        \end{claim}
        \begin{proof} Since $h$ is univalent, it does not have an essential singularity at $0$. We extend $h$ to $0$ so that it is a meromorphic function. Let $\gamma(t)=\exp(2\pi i(t+iR))$ for $t\in[0,1]$, and $R>0$ large. Then $\frac{h'(\exp(2\pi iz))}{h(\exp(2\pi iz))}\exp(2\pi iz)=g'(z)$, and hence 
        \begin{align*}
            \frac{1}{2\pi i}\int_\gamma \frac{h'(z)}{h(z)}dz&=\int_0^1 g'(t+iR)dt=1.
        \end{align*}
        This shows that $h$ has a simple zero at $0$. Hence $h(0)=0$.
        \end{proof}
        \par We can now apply Claim \ref{claim:de-branges} to get 
        \begin{align*}
            \abs{h(z)-h'(0)z}<100 \abs{h'(0)}\abs{z}^2,
        \end{align*}
        whenever $\abs{z}<\frac{1}{2}$. We now set $C>\frac{\log 2}{2\pi}$, so that when $\mathrm{Im}(z)\geq C$, we have $\abs{\exp(2\pi iz)}<\frac{1}{2}$, and hence 
        \begin{align*}
            \abs{\frac{1}{h'(0)}\exp(2\pi ig(z))-\exp(2\pi iz)}<100 e^{-4\pi\mathrm{Im}(z)}.
        \end{align*}
        Therefore 
        \begin{align*}
            \abs{\exp(2\pi i(g(z)-z+\alpha))-1}<100 e^{-2\pi\mathrm{Im}(z)},
        \end{align*}
        where $\alpha\in\mathbb{C}$ is such that $\exp(-2\pi i\alpha)=h'(0)$. For $C$ large enough, whenever $\mathrm{Im}(z)\geq C$, we will have $\abs{g(z)-z+\alpha-m}\leq \frac{1}{10}$ for some integer $m\in\mathbb{Z}$. We replace $\alpha$ with $\alpha-m$, so that $\abs{g(z)-z+\alpha}\leq \frac{1}{10}$. Since $\exp$ is bilipschitz on the disk centered at $0$ of radius $\frac{1}{10}$, we see that 
        \begin{align*}
            \abs{g(z)-z+\alpha}\lesssim e^{-2\pi\mathrm{Im}(z)},
        \end{align*}
        and (\ref{item:g-est}) follows.
        \par To show (\ref{item:g'-est}), we note that for $w=\exp(2\pi iz)$, 
        \begin{align*}
            g'(z)=w\frac{h'(w)}{h(w)}.
        \end{align*}
        By the Koebe distortion theorem, since $h$ is univalent on the unit disk, we have 
        \begin{align*}
            \frac{1-\abs{w}}{1+\abs{w}}\leq\abs{w\frac{h'(w)}{h(w)}}\leq \frac{1+\abs{w}}{1-\abs{w}},
        \end{align*}
        so for any $C>0$ fixed, $\abs{w}=e^{-2\pi\mathrm{Im}(z)}\leq e^{-2\pi C}$, and we have $\abs{g'(z)}=\abs{w\frac{h'(w)}{h(w)}}\approx 1$.
        \par We now turn to (\ref{item:koebe-quarter}). Note that $\mathrm{Im}(\alpha)=\frac{1}{2\pi}\mathrm{Re}(-2\pi i\alpha)=\frac{1}{2\pi}\log\abs{h'(0)}$. By the Koebe quarter theorem, the set $h(\mathbb{D})$ contains the disk $D$ centered at $0$ with radius $\frac{1}{4}\abs{h'(0)}=\frac{1}{4}e^{2\pi\mathrm{Im}(\alpha)}$. Note that $\exp(2\pi iz)\in D$ if and only if $\mathrm{Im}(z)\geq -\mathrm{Im}(\alpha)+\frac{1}{2\pi}\log 4=\frac{\log 2}{\pi}-\mathrm{Im}(\alpha)$, so the final claim of the Proposition follows.
    \end{proof}
    \subsection{Derivative bounds on a horoball}\label{subseciton:derivative-bound}
    Here we give a lower bound on the derivative $\abs{f'}$ of a univalent map $f:\mathbb{D}\to\mathbb{C}$ on a horoball, assuming that $f$ conjugates a parabolic isometry of $\mathbb{D}$ that preserves that horoball to a parabolic M\"obius transformation. The main result is Lemma \ref{lm:f'-estimate}. This lemma is used to bound the integral $\int \abs{f'}^{-q}dxdy$ over a horoball in Corollary \ref{cor:estimate-horoball-integral}. This corollary is the result we use in the coming sections. 
    \par We first define the horoballs we will consider. All horoballs will be open. Given a horoball $H$, the horocycle $\partial H$ admits a natural orientation as follows. The vector $v$ along $\partial H$ at $z\in\partial H$ is positive if $(v, w)$ is a positively-oriented basis of $T_z\mathbb{D}$, where $w$ is the vector at $z$ tangent to the geodesic ray from $z$ to the point at infinity of $H$.
    \begin{dfn}\label{dfn:horoball}
        A horoball $H\subset\mathbb{D}$ is $\ell$-adapted to a parabolic isometry $\gamma\in\mathrm{Aut}(\mathbb{D})$ if the following conditions hold,
        \begin{enumerate}
            \item the parabolic $\gamma$ preserves $H$,
            \item the distance $\sup_{z\in H}\mathrm{dist}(z, \gamma(z))=\ell$,
            \item the parabolic $\gamma$ moves points on $\partial H$ in the positive direction.
        \end{enumerate}  
        When $0\not\in H$, we define the anchor of $H$ to be the point $z\in\partial H$ closest to the origin in the hyperbolic metric.
    \end{dfn}
    \begin{dfn}\label{dfn:parabolic-pm}
        We say that a parabolic $\gamma\in\mathrm{Aut}(\mathbb{D})$ is positive if $\gamma$ moves points on $\partial H$ in the positive direction, for some horoball $H$ that is preserved by $\gamma$. We say that $\gamma$ is negative if it is not positive.
    \end{dfn}
    We note that $\gamma\in\mathrm{Aut}(\mathbb{D})$ is a positive parabolic if and only if $\gamma^{-1}$ is a negative parabolic. It is clear from the definitions that only positive parabolics in $\mathrm{Aut}(\mathbb{D})$ can be adapted to horoballs.
    \begin{remark}\label{rmrk:adapted-horoballs-exist}
        When $\gamma(z)=z+1$ in the upper half-plane model, the horoballs $\{z:\mathrm{Im}(z)\geq C\}$ are $\mathrm{dist}(iC, iC+1)$-adapted. The function $C\to\mathrm{dist}(iC, iC+1)$ is decreasing, converges to infinity as $C\to 0$ and to $0$ as $C\to\infty$, so by continuity $\ell$-adapted horoballs exist and are unique for all $\ell>0$. Any positive parabolic is conjugate to $\gamma$, so the analogous picture holds for an arbitrary positive parabolic. 
    \end{remark}
    Our main result is the following. 
    \begin{lm}\label{lm:f'-estimate}
There exist universal constants $L, C>0$ such that the following holds. Let $f:\mathbb{D}\to\mathbb{C}$ be a univalent map, and let $H$ be a horoball not containing $0$ that is $\ell$-adapted to a parabolic $\gamma\in\mathrm{Aut}(\mathbb{D})$, with $\ell<L$. Let the anchor of $H$ be $z_0$. We assume that $f$ conjugates $\gamma$ to a parabolic M\"obius transformation. Then for all $z\in H$,
        \begin{align*}
            \abs{f'(z)}\geq C\abs{f'(z_0)}.
        \end{align*}
    \end{lm}
    \begin{remark}
        It is crucial that $f(\mathbb{D})$ does not contain infinity in Lemma \ref{lm:f'-estimate}. Without this assumption, the Lemma does not hold. Note that we allow $f$ to be unbounded, or equivalently $\infty\in\partial f(\mathbb{D})$.
    \end{remark}
    \begin{cor}\label{cor:estimate-horoball-integral}
        Let $L, C, \ell$ be as in Lemma \ref{lm:f'-estimate}. Let $H$ be a horoball not containing $0$ that is $\ell$-adapted to a parabolic $\gamma\in\Gamma$, with anchor $z_0\in\partial H$. Then for any $q>0$,
        \begin{align*}
            \int_H \frac{1}{\abs{f'(z)}^{q}}dxdy\leq\frac{\pi}{4C^q} \frac{\left(1-\abs{z_0}^2\right)^2}{\abs{f'(z_0)}^q}.
        \end{align*}
    \end{cor}
    \begin{proof}
        It is standard that $f$ conjugates $\gamma$ to another parabolic $\rho(\gamma)$. From Lemma \ref{lm:f'-estimate} we see that 
        \begin{align*}
            \int_H \frac{1}{\abs{f'(z)}^q}dxdy&\leq \frac{1}{C^q\abs{f'(z_0)}^q} \int_H dxdy=\frac{1}{C^q\abs{f'(z_0)}^q} \frac{\pi(1-\abs{z_0})^2}{4}\\
            &\leq \frac{\pi}{4C^q}\frac{\left(1-\abs{z_0}^2\right)^2}{\abs{f'(z_0)}^q}.
        \end{align*}
    \end{proof}
    \begin{proof}[Proof of Lemma \ref{lm:f'-estimate}]
        The desired inequality will follow from the estimates in Proposition \ref{prop:koebe-parabolic}. We first need to conjugate $f$ to a map $\mathbb{H}\to\mathbb{C}$ that commutes with $z\to z+1$. 
        \par Suppose without loss of generality that $\gamma$ fixes $1$. We will construct a commuting square
        \[
        \begin{tikzcd}
        \mathbb{H} \arrow{r}{g} \arrow[swap]{d}{A} & \mathbb{C} \arrow{d}{B} \\
        \mathbb{D} \arrow{r}{f} & \mathbb{C}
        \end{tikzcd},
        \]
        where $A$ and $B$ are M\"obius transformations, and $g(z+1)=g(z)+1$.
        \par We let $A(z)=\frac{ z-\lambda i}{z+\lambda i}$ for $\lambda>0$ chosen such that $A^{-1}\circ\gamma\circ A(z)=z+1$. We choose $B$ depending on finiteness of the fixed point of $f\circ\gamma\circ f^{-1}$. 
        \par We first give some preliminary observations that will be useful in both cases. The following Claim relates $f'$ and $g'$.
        \begin{claim}\label{claim:lm:deriv-bound-prelim-claim}
            We have for all $z$,
            \begin{align*}
                (f'\circ A) (z)=(B'\circ g)(z) g'(z) \frac{ (z+i\lambda)^2}{2\lambda i}.
            \end{align*}
        \end{claim} 
        \begin{proof}
            We have $f=B\circ g\circ A^{-1}$, and $A^{-1}(z)=\lambda i\frac{1+z}{1-z}$. Therefore 
            \begin{align*}
                f'(A(z))=B'(g(z)) g'(z) \left(A^{-1}\right)'(A(z)).
            \end{align*}
            We have $(A^{-1})'(A(z))=\frac{2i\lambda }{\left(1-\frac{z-\lambda i}{z+\lambda i}\right)^2}=\frac{(z+\lambda i)^2}{2\lambda i}$, and the result follows.
        \end{proof}
        For $z\in H$, we write $\textbf{z}=A^{-1}(z)$. In particular $\textbf{z}_0=A^{-1}(z_0)$.
        \begin{claim}\label{claim:bold-z-z0}
            We have $\mathrm{Re}(\textbf{z}_0)=0$ and $\mathrm{Im}(\textbf{z})\geq\mathrm{Im}(\textbf{z}_0)\geq\lambda$ for all $z\in H$.
        \end{claim}
        \begin{proof}
            We have $A(\infty)=1$ and $A(\lambda i)=0$. Since the geodesic ray $[0,1)$ contains $z_0$, the geodesic ray $[\lambda i, \infty)$ contains $\textbf{z}_0$. In particular $\mathrm{Re}(\textbf{z}_0)=0$ and $\mathrm{Im}(\textbf{z}_0)\geq\lambda$.
            \par Since $H$ is a horoball containing $1$, the horoball $A^{-1}(H)$ contains infinity. Since $\textbf{z}_0\in\partial A^{-1}(H)$, the second claim follows.
        \end{proof}
        \begin{enumerate}
            \item[Case 1] We first assume that $f\circ\gamma\circ f^{-1}$ fixes infinity. We let $B(z)=az$ for $a\in\mathbb{C}\setminus\{0\}$ chosen so that $B^{-1}\circ (f\circ\gamma\circ f^{-1})\circ B(z)=z+1$. Since $A$ and $B$ are invertible, the commuting square exists by setting $g=B^{-1}\circ f\circ A$. Then $f\circ A$ conjugates $z\to z+1$ to $f\circ\gamma\circ f^{-1}$, and hence $B^{-1}\circ f\circ A=g$ conjugates $z\to z+1$ to itself. In particular $g(z+1)=g(z)+1$.
            \par It follows from Claim \ref{claim:lm:deriv-bound-prelim-claim} that 
            \begin{align*}
                \abs{f'(z)}=\abs{\frac{a}{2\lambda}} \abs{g'(\textbf{z})} \abs{\textbf{z}+\lambda i}^2.
            \end{align*}
            By Claim \ref{claim:bold-z-z0}, we have 
            \begin{align*}
                \frac{\abs{f'(z)}}{\abs{f'(z_0)}}=\frac{\abs{g'(\textbf{z})}}{\abs{g'(\textbf{z}_0)}} \frac{\abs{\textbf{z}+\lambda i}^2}{\abs{\textbf{z}_0+\lambda i}^2}\geq \frac{\abs{g'(\textbf{z})}}{\abs{g'(\textbf{z}_0)}}\approx 1,
            \end{align*}
            where we used Proposition \ref{prop:koebe-parabolic} in the second estimate. 
            \item[Case 2] Assume now that $f\circ\gamma\circ f^{-1}$ fixes a point $a\in\mathbb{C}$. Set $B(z)=a+\frac{b}{z}$, where $b\in\mathbb{C}\setminus\{0\}$ is chosen so that $B^{-1}\circ (f\circ\gamma\circ f^{-1})\circ B(z)=z+1$. As in Case 1, both $A$ and $B$ are invertible so the commuting square exists, and $g(z+1)=g(z)+1$.
            \par By Claim \ref{claim:lm:deriv-bound-prelim-claim},
            \begin{align*}
                \abs{f'(A(z))}=\abs{\frac{b}{2\lambda}} \frac{1}{\abs{g(z)}^2}\abs{g'(z)} \abs{z+\lambda i}^2.
            \end{align*}
            We have 
            \begin{align*}
                \frac{\abs{f'(z)}}{\abs{f'(z_0)}}=\abs{\frac{g(\textbf{z}_0)}{g(\textbf{z})}}^2 \abs{\frac{g'(\textbf{z})}{g'(\textbf{z}_0)}} \abs{\frac{\textbf{z}+\lambda i}{\textbf{z}_0+\lambda i}}^2.
            \end{align*}
            By Proposition \ref{prop:koebe-parabolic}, $\abs{\frac{g'(\textbf{z})}{g'(\textbf{z}_0)}}\approx 1$ and there exists $\alpha\in\mathbb{C}$ such that
            \begin{align*}
                \abs{g(\textbf{z})-\textbf{z}+\alpha}\lesssim e^{-2\pi \mathrm{Im}(\textbf{z})}.
            \end{align*}
            Since infinity is not in the image of $f$, $0$ is not in the image of $g$. In particular by Proposition \ref{prop:koebe-parabolic}, $\frac{\log 2}{\pi}-\mathrm{Im}(\alpha)\geq 0$. For $\ell$ small enough, which corresponds to $\mathrm{Im}(\textbf{z}_0)$ large enough, we have $\mathrm{Im}(\textbf{z}-\alpha)\geq\mathrm{Im}(\textbf{z}_0)-\frac{\log 2}{\pi}\gtrsim e^{-2\pi\mathrm{Im}(\textbf{z}_0)}$, and hence $\abs{g(\textbf{z})}\approx\abs{\textbf{z}-\alpha}$. Therefore
            \begin{align*}
                \frac{\abs{f'(z)}}{\abs{f'(z_0)}}=\left.\abs{\frac{\textbf{z}_0-\alpha}{\textbf{z}_0+\lambda i}}^2\middle/\abs{\frac{\textbf{z}-\alpha}{\textbf{z}+\lambda i}}^2\right..
            \end{align*}
            Let $L$ be small enough so that for all $\textbf{z}\in A^{-1}(H)$, we have $\mathrm{Im}(\textbf{z})\geq 1>10\frac{\log 2}{\pi}>10\mathrm{Im}(\alpha)$. Hence  
            \begin{align*}
                \abs{\frac{\textbf{z}-\alpha}{\textbf{z}+\lambda i}}^2\approx \frac{(\mathrm{Re}(\textbf{z})-\mathrm{Re}(\alpha))^2+\mathrm{Im}(\textbf{z})^2}{\abs{\textbf{z}}^2+\lambda^2}\lesssim \frac{\abs{\textbf{z}}^2+\mathrm{Re}(\alpha)^2}{\abs{\textbf{z}}^2}.
            \end{align*}
            Similarly we have 
            \begin{align*}
                \abs{\frac{\textbf{z}_0-\alpha}{\textbf{z}_0+\lambda i}}^2&\approx \frac{\mathrm{Re}(\alpha)^2+\mathrm{Im}(\textbf{z}_0)^2}{\mathrm{Im}(\textbf{z}_0)^2+{\lambda}^2}\approx \frac{\mathrm{Re}(\alpha)^2+\mathrm{Im}(\textbf{z}_0)^2}{\mathrm{Im}(\textbf{z}_0)^2}\geq 1+\frac{\mathrm{Re(\alpha)^2}}{\abs{\textbf{z}}^2}\gtrsim\abs{\frac{\textbf{z}-\alpha}{\textbf{z}+\lambda i}}^2,
            \end{align*}
            where we used $\mathrm{Im}(\textbf{z}_0)\geq \lambda>0$ in the second estimate and $\abs{\textbf{z}}^2\geq\mathrm{Im}(\textbf{z})^2\geq\mathrm{Im}(\textbf{z}_0)^2$ in the third estimate. Therefore $\frac{\abs{f'(z)}}{\abs{f'(z_0)}}\gtrsim 1$, as desired.
        \end{enumerate}
    \end{proof}
    \subsection{Dividing the disk into adapted horoballs and a cocompact subset}\label{subsection:division}
    Recall that to show Theorem \ref{thm:main-estimate}, we split the domain $\mathbb{D}$ into horoballs on which we use Lemma \ref{lm:f'-estimate}, and into the orbit of a compact set. We describe this splitting here, and how to estimate the integral from Theorem \ref{thm:main-estimate} over translates of a compact set in the next subsection.
    \par Let $L>0$ be the constant from Lemma \ref{lm:f'-estimate}. Choose $\ell$ arbitrarily with $0<\ell<\min\left(L, \inf_{\gamma\in\Gamma}\mathrm{dist}(0,\gamma(0))\right)$.
    Let $\mathcal{H}$ be the union over all positive parabolic $\gamma\in\Gamma$ of the (open) horoball $H_\gamma$ that is $\ell$-adapted to $\gamma$. We let $\Phi$ be the closed Dirichlet fundamental region for $\Gamma$ centered at $0$, and write $\Phi^*=\Phi\setminus\mathcal{H}$. 
    \begin{prop}\label{prop:properties-splitting}
        The set $\Phi^*$ is compact with $0$ in its interior and is a fundamental domain for the action of $\Gamma$ on $\mathbb{D}\setminus\mathcal{H}$.
    \end{prop}
    \begin{proof}
        Implicit in the statement of the proposition is the claim that $\mathcal{H}$ is $\Gamma$-invariant. This follows from uniqueness of $\ell$-adapted horoballs that was explained in Remark \ref{rmrk:adapted-horoballs-exist}. The final claim follows from this and the fact that $\Phi$ is a fundamental domain for the action of $\Gamma$ on $\mathbb{D}$.
        \par Since $\mathbb{D}/\Gamma$ has finite area, by Siegel's theorem the Dirichlet fundamental domain is a convex hyperbolic polygon \cite[Theorem 4.1.1]{katok}. It is well known that any vertex at infinity $v\in\overline{\Phi}\cap\mathbb{S}^1$ of this polygon is a fixed point of a parabolic $\gamma_v\in\Gamma$ \cite[Theorem 9.4.5 (4)]{beardon}. By replacing $\gamma_v$ with $\gamma_v^{-1}$ if necessary, we may assume that $\gamma_v$ is positive. Therefore $\mathcal{H}$ contains horoballs $H^v:=H_{\gamma_v}$ that are $\ell$-adapted to $\gamma_v$, and hence based at $v$, for each $v\in\overline{\Phi}\cap\mathbb{S}^1$. Therefore 
        \begin{align*}
            \Phi^*=\Phi\setminus\mathcal{H}\subseteq\Phi\setminus\bigcup_{v\in \overline{\Phi}\cap\mathbb{S}^1}H^v
        \end{align*}
        is bounded, and thus compact.
        \par For any positive parabolic $\gamma\in\Gamma$, since $\mathrm{dist}(0,\gamma(0))>\ell$, we have $0\not\in H_\gamma$. Moreover since $\inf_{\gamma\in\Gamma}\mathrm{dist}(0,\gamma(0))>\ell$, we have 
        \begin{align*}
            \inf_{\gamma\in\Gamma}\mathrm{dist}(0, H_\gamma)>0.
        \end{align*}
        Therefore $0$ lies in the interior of $\Phi^*=\Phi\setminus\bigcup_\gamma H_\gamma$.
    \end{proof}
    \subsection{Integral estimates on the lift of a compact part}\label{subsection:compact-piece}
    Recall that one of the steps in the proof of Theorem \ref{thm:main-estimate} is bounding $\int_{\mathbb{D}\setminus\mathcal{H}}\abs{f'}^{-q}dxdy$. We do this by splitting \[\int_{\mathbb{D}\setminus\mathcal{H}}\frac{1}{\abs{f'(z)}^q}dxdy=\sum_{\gamma\in\Gamma}\int_{\gamma\Phi^*}\frac{1}{\abs{f'(z)}^q}dxdy.\]
    Our goal in this subsection is to show how $\int_{\gamma\Phi^*}\abs{f'}^{-q}dxdy$ depends on $\gamma\in\Gamma$. This follows from the Koebe distortion theorem that guarantees that $\abs{f'}$ changes at most by a constant factor over $\gamma\Phi^*$.
    \begin{prop}\label{prop:compact-piece}
        For any $q>0$ there exists a constant $C=C(\Gamma,q)$ such that
        \begin{align*}
            \frac{1}{C}\frac{\left(1-\abs{\gamma(0)}^2\right)^2}{\abs{f'(\gamma(0))}^q}\leq\int_{\gamma\Phi^*} \frac{1}{\abs{f'(z)}^q}dxdy\leq C \frac{\left(1-\abs{\gamma(0)}^2\right)^2}{\abs{f'(\gamma(0))}^q}
        \end{align*}
        for all $\gamma\in\Gamma$.
    \end{prop}
    \begin{proof}
        By Proposition \ref{prop:properties-splitting}, $\Phi^*$ is a compact set whose interior contains $0$, so we can pick radii $r=r(\Gamma)<R=R(\Gamma)$ that depend only on $\Gamma$ such that $B(0, r)\subseteq\Phi^*\subseteq B(0, R)$. Here we denote by $B(z, C)$ the hyperbolic disk centered at $z$ of radius $C$. 
        \par By the Koebe distortion theorem, we have for $z\in\gamma\Phi^*$,
        \begin{align*}
            e^{-2R(\Gamma)}\leq e^{-2\mathrm{dist}(\gamma(0), z)}\leq\abs{\frac{f'(\gamma(0))(1-\abs{\gamma(0)}^2)}{f'(z)(1-\abs{z}^2)}}\leq e^{2\mathrm{dist}(\gamma(0), z)}\leq e^{2R(\Gamma)}.
        \end{align*}   
        In particular, we have $\abs{f'(z)}(1-\abs{z}^2)\approx_\Gamma\abs{f'(\gamma(0))}(1-\abs{\gamma(0)}^2)$. 
        \par Since $\gamma\Phi^*\subseteq B(\gamma(0), R)$, and the Euclidean diameter of the hyperbolic disk $B(\gamma(0), R)$ is $R(1-\abs{\gamma(0)}^2)$ up to a bounded multiplicative error, on $\gamma\Phi^*$, we have $1-\abs{z}^2\approx 1-\abs{\gamma(0)}^2$. In particular, $\abs{f'(z)}\approx_\Gamma \abs{f'(\gamma(0))}$ for $z\in\gamma\Phi^*$. 
        \par Therefore 
        \begin{align*}
            \int_{\Phi^*} \frac{1}{\abs{f'(z)}^q}dxdy\approx_{\Gamma,q} \frac{1}{\abs{f'(\gamma(0))}^q} \int_{\gamma\Phi^*}dxdy.
        \end{align*}
        The region $\gamma\Phi^*$ contains a hyperbolic disk of radius $r$, and hence a Euclidean disk of radius $r'\approx r(1-\abs{\gamma(0)}^2)$. Therefore 
        \begin{align*}
            r^2(1-\abs{\gamma(0)}^2)^2\lesssim\int_{\gamma\Phi^*}dxdy\lesssim R^2 (1-\abs{\gamma(0)}^2)^2,
        \end{align*}
        and hence $\int_{\gamma\Phi^*}dxdy\approx_\Gamma (1-\abs{\gamma(0)}^2)^2$. Hence
        \begin{align*}
            \int_{\Phi^*} \frac{1}{\abs{f'(z)}^q}dxdy\approx_{\Gamma,q} \frac{\left(1-\abs{\gamma(0)}^2\right)^2}{\abs{f'(\gamma(0))}^q}.
        \end{align*} 
    \end{proof}
    \subsection{Proof of Theorem \ref{thm:main-estimate}}\label{subsection:pf-of-thm}
    We note that $f(\gamma(z))=\rho(\gamma)(f(z))$. Differentiating at $0$, we see that $f'(\gamma(0))\gamma'(0)=\rho(\gamma)'(0) f'(0)$. Therefore 
    \begin{align}\label{eq:f'-equiv}
        f'(\gamma(0))=f'(0)\frac{\rho(\gamma)'(0)}{\gamma'(0)}.
    \end{align} 
    \par Since $\gamma$ is a M\"obius transformation of the disk, we can write it as $\gamma(z)=\lambda \frac{z-a}{1-\bar{a}z}$, where $a\in\mathbb{D}$ and $\abs{\lambda}=1$. Therefore 
    \begin{align}\label{eq:gamma'}
        \abs{\gamma'(0)}=1-\abs{a}^2=1-\abs{\gamma(0)}^2.
    \end{align}
    Using (\ref{eq:f'-equiv}), (\ref{eq:gamma'}) and Proposition \ref{prop:compact-piece}, we find that 
    \begin{align}\label{eq:combined-compact}
        \int_{\mathbb{D}\setminus\mathcal{H}} \frac{1}{\abs{f'(z)}^q}dxdy=\sum_{\gamma\in\Gamma}\int_{\gamma\Phi^*}\frac{1}{\abs{f'(z)}^q}dxdy\approx_{\Gamma,q}\sum_{\gamma\in\Gamma} \frac{\abs{\gamma'(0)}^{q+2}}{\abs{\rho(\gamma)'(0)}^q}.
    \end{align}
    \par For any positive parabolic $\gamma\in\Gamma$, by Proposition \ref{prop:properties-splitting} we have $0\not\in H_\gamma$. Fix such a $\gamma\in\Gamma$, and let $z_0$ be the anchor of $H_\gamma$. Then $z_0\in\gamma\Phi^*$ for some $\gamma\in\Gamma$. Hence by Corollary \ref{cor:estimate-horoball-integral} we have 
    \begin{align*}
        \int_{H_\gamma} \frac{1}{\abs{f'(z)}^q}dxdy\lesssim_q \frac{\left(1-\abs{z_0}^2\right)^2}{\abs{f'(z_0)}^q}.
    \end{align*}
    Since $\gamma\Phi^*$ has hyperbolic diameter bounded over $\gamma\in\Gamma$, we have by the Koebe distortion theorem 
    \begin{align*}
        \abs{\frac{f'(\gamma(0))(1-\abs{\gamma(0)}^2)}{f'(z_0)(1-\abs{z_0}^2)}}\approx_\Gamma 1.
    \end{align*}
    We also have as in the proof of Proposition \ref{prop:compact-piece}, 
    \begin{align*}
        1-\abs{z_0}^2\approx_{\Gamma} 1-\abs{\gamma(0)}^2.
    \end{align*}
    Therefore 
    \begin{align*}
        \int_{H_\gamma} \frac{1}{\abs{f'(z)}^q}dxdy\lesssim_{\Gamma,q} \frac{\left(1-\abs{\gamma(0)}^2\right)^2}{\abs{f'(\gamma(0))}^q}\approx_{\Gamma,q} \int_{\gamma\Phi^*} \frac{1}{\abs{f'(z)}^q}dxdy.
    \end{align*}
    Since any fundamental region $\gamma\Phi^*$ intersects at most a bounded number of horoballs in $\mathcal{H}$, we see that 
    \begin{align}\label{eq:combined-horoball}
        \int_{\mathcal{H}}\frac{1}{\abs{f'(z)}^q}dxdy\lesssim_{\Gamma,q} \int_{\mathbb{D}\setminus\mathcal{H}} \frac{1}{\abs{f'(z)}^q}dxdy.
    \end{align}
    Combining (\ref{eq:combined-compact}) and (\ref{eq:combined-horoball}), we see that 
    \begin{align*}
        \int_{\mathbb{D}}\frac{1}{\abs{f'(z)}^q}dxdy\approx_{\Gamma,q}\int_{\mathbb{D}\setminus\mathcal{H}} \frac{1}{\abs{f'(z)}^q}dxdy\approx_{\Gamma,q}\sum_{\gamma\in\Gamma} \frac{\abs{\gamma'(0)}^{q+2}}{\abs{\rho(\gamma)'(0)}^q}.
    \end{align*}
    \section{Grafting a once-punctured torus}\label{section:grafting}
    In this section and in the rest of the paper, denote by $\Omega$ the domain from \cite{komori}, and by $\Gamma$ the Kleinian once-punctured torus group that acts on $\Omega$. Here we give a construction of $\Omega$ and $\Gamma$.
    \par In \cite{komori}, $\Omega$ is obtained by starting with a hyperplane in $\mathbb{H}^3$ and bending it along a certain discrete set of geodesics. The result of this is a set in $\mathbb{H}^3$ that disconnects $\partial\mathbb{H}^3=\mathbb{S}^2$, and $\Omega$ is one of the resulting connected components. Since here we do not work with hyperbolic space $\mathbb{H}^3$, we give a concrete description of $\Omega$ as the union of regions in the complex plane bounded by 4 or 2 circular arcs. 
    \par In this section we work in the upper half-plane, closely following \cite[\S 2]{komori}. We denote the M\"obius transformations and subsets of $\mathbb{H}$ with a bar to distinguish them from their counterparts in the disk model.
    \par We start with a once-punctured torus group generated by M\"obius transformations
    \begin{align*}
        \bar{A}=\begin{pmatrix}
            \cosh \frac{\lambda}{2} & \cosh\frac{\lambda}{2} + 1\\
            \cosh\frac{\lambda}{2} - 1& \cosh\frac{\lambda}{2}
        \end{pmatrix},\\
        \bar{B}=\begin{pmatrix}
            \cosh\frac{\tau}{2}\coth\frac{\lambda}{4} & -\sinh\frac{\tau}{2}\\
            -\sinh\frac{\tau}{2} & \cosh\frac{\tau}{2}\tanh\frac{\lambda}{4}
        \end{pmatrix}.
    \end{align*}
    When $\lambda$ and $\tau$ are real, the group $\langle \bar{A},\bar{B}\rangle$ acts on $\mathbb{H}$ and $\frac{\mathbb{H}}{\langle \bar{A},\bar{B}\rangle}$ is a once-punctured torus. We fix $\lambda=2\cosh^{-1}\frac{3}{2}$ throughout, and we let $\tau=\frac{\lambda}{2}+i\theta$ with $\theta$ a real parameter, to be set later. We denote by $\bar{B}(\theta)$ the M\"obius transformation obtained by setting $\tau=\frac{\lambda}{2}+i\theta$ in the formula above, and set $\bar{\Gamma}(\theta)=\langle \bar{A}, \bar{B}(\theta)\rangle$. We write $\bar{\Gamma}=\bar{\Gamma}(0)$. Define a natural homomorphism $\bar{f}_\theta:\bar{\Gamma}\to\bar{\Gamma}(\theta)$ by $\bar{f}_\theta(\bar{A})=\bar{A}, \bar{f}_\theta(\bar{B}(0))=\bar{B}(\theta)$.
    \par Let $\Phi$ be the ideal quadrilateral with vertices $\pm\tanh\frac{\lambda}{4}, \pm\coth\frac{\lambda}{4}$, so that $\Phi$ is a fundamental domain for the action of $\bar{\Gamma}$ on $\mathbb{H}$. Let $\Sigma$ be the union of axes of all elements of $\bar{\Gamma}$. Then $\mathbb{H}\setminus\Sigma$ is a disjoint union of $\gamma \mathrm{Int}(\Phi)$, where $\mathrm{Int}$ denotes the topological interior. 
    \par Define a map $\bar{\psi}_\theta:\mathbb{H}\setminus\Sigma\to\mathbb{S}^2=\mathbb{C}\cup\{\infty\}$ by 
    \begin{align*}
        \bar{\psi}_\theta(z)=\bar{f}_\theta(\gamma)(\gamma^{-1}z),
    \end{align*} 
    where $\gamma\in\bar{\Gamma}$ is the unique element with the property that $z\in \gamma \mathrm{Int}(\Phi)$.
    \par For any hyperbolic element $\gamma\in\bar{\Gamma}$, the axis of $\gamma$ is a boundary of exactly two (ideal) quadrilaterals $\eta_1\Phi,\eta_2\Phi$ of the form $\eta\Phi$ for $\eta\in\bar{\Gamma}$. Order $\eta_1,\eta_2$ so that $\eta_1^{-1}\eta_2\in\{\bar{A},\bar{B}\}$. Then $\bar{\psi}_\theta(\eta_2\Phi)=\bar{f}_\theta(\eta_1)\bar{f}_\theta(\eta_1^{-1}\eta_2)\Phi$ and $\bar{\psi}_\theta(\eta_1\Phi)=\bar{f}_\theta(\eta_1)\Phi$. Since $\bar{f}_\theta(\bar{A})=\bar{A}$ we can extend $\bar{\psi}_\theta$ over all axes in $\Sigma$ that separate $\eta_1\Phi$ and $\eta_1 \bar{A}\Phi$. Denote the union of all other axes in $\Sigma$ by $\Sigma_B$. For $\theta<0$ there is a region $G_\theta$ between $\Phi$ and $\bar{f}_\theta(B)\Phi$ bounded by two circular arcs (see Figure \ref{fig:gap}).    
    \begin{figure}
        \begin{center}
            \includegraphics[width=\textwidth, keepaspectratio]{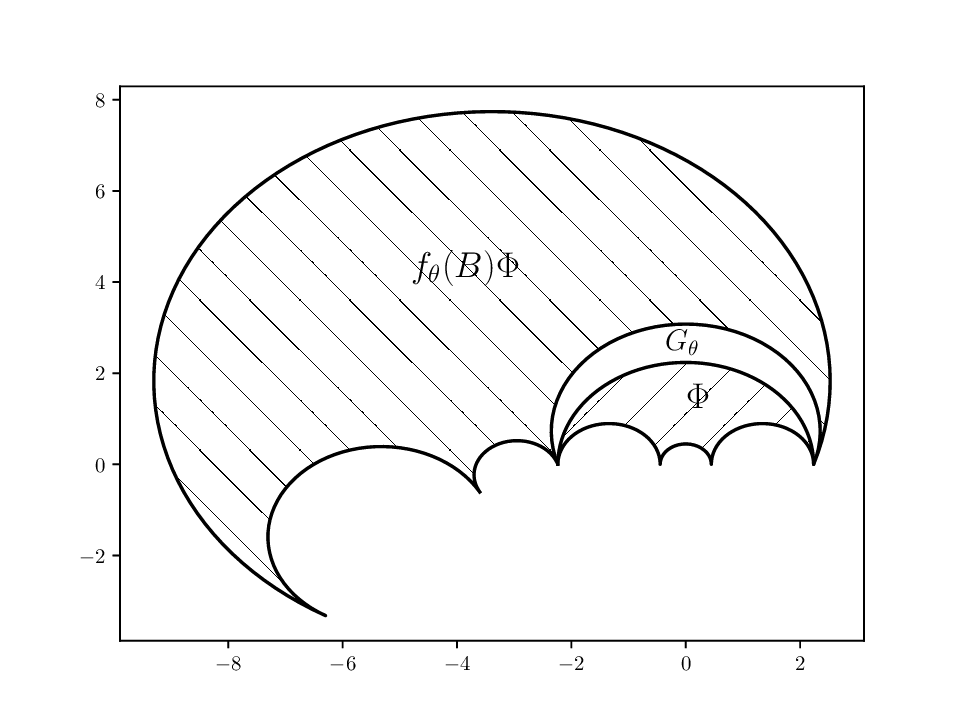}
        \end{center}
        \caption{The gap $G_\theta$ between $\Phi=\bar{f}_\theta(\mathrm{id})\Phi$ and $\bar{f}_\theta(B)\Phi$ in the image of $\bar{\psi}_\theta$. }\label{fig:gap}
    \end{figure} 

    \begin{align*}
        \bar{\Omega}(\theta)=\bar{\psi}_\theta(\mathbb{H}\setminus\Sigma_B)\cup\bigcup_{\gamma\in\bar{\Gamma}(\theta)} \gamma G_\theta.
    \end{align*}
    Then $\bar{\Omega}(\theta)$ is simply connected and $\bar{\Gamma}(\theta)$-invariant.
    \par Fix $\eta(z)=i\frac{1+z}{1-z}$ to be a M\"obius transformation that maps $\mathbb{D}$ to $\mathbb{H}$. Set $\theta_0=\arccos\frac{1}{9}-\pi$. 
    \begin{dfn}\label{dfn:omega-gamma}
        We define $\Omega=\eta^{-1}(\bar{\Omega}(\theta_0))$, and $\Gamma=\eta^{-1}\circ\bar{\Gamma}(\theta_0)\circ\eta$.
    \end{dfn}
    We also set $f(\gamma)=\eta^{-1}\circ\bar{f}_{\theta_0}(\gamma)\circ\eta$ and $\psi=\eta^{-1}\circ\bar{\psi}_{\theta_0}$.
    A picture of $\Omega$ can be found in Figure \ref{fig:domain}.
    \begin{figure}
        \begin{center}
        \includegraphics[width=\textwidth, keepaspectratio]{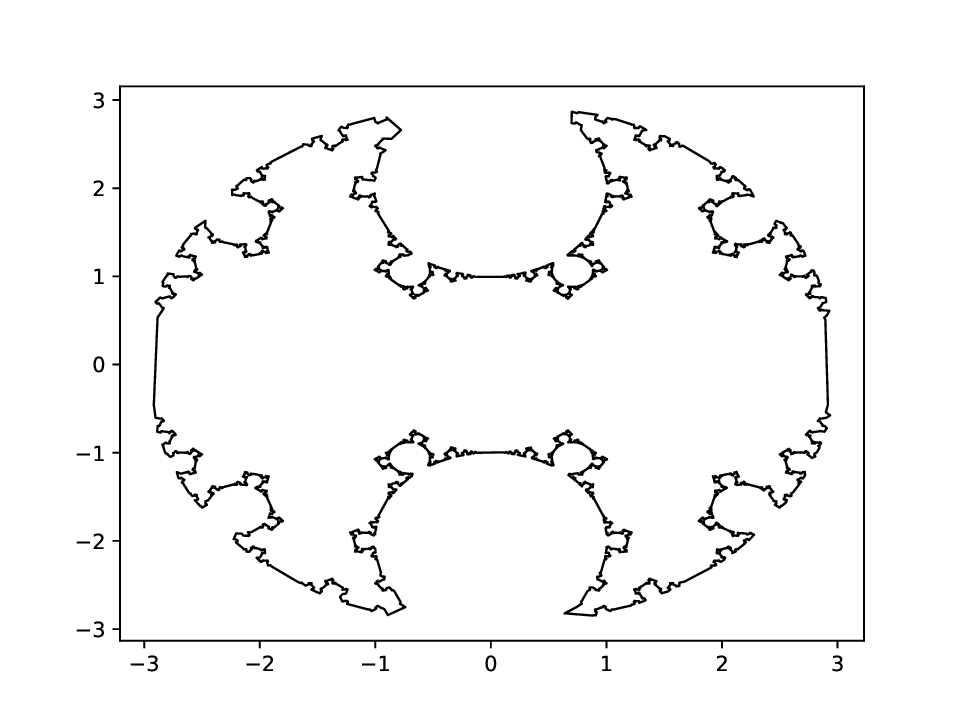}
        \end{center}
        \caption{Domain $\Omega$.}\label{fig:domain}
    \end{figure}
    \section{Numerical results}\label{section:methods}
    Let $F:\Omega\to\mathbb{D}$ be the Riemann map and $f=F^{-1}$ be its inverse. We denote by $\Gamma_0=F\circ\Gamma \circ F^{-1}$ and by $\rho:\Gamma_0\to\Gamma$ the homomorphism induced by $f$. 
    \par By a simple change of coordinates we see that 
    \begin{align}\label{eq:change-of-coords}
        \int_\Omega \abs{F'}^p dxdy=\int_\mathbb{D}\abs{f'}^{2-p}dxdy.
    \end{align}
    By Theorem \ref{thm:main-estimate} and (\ref{eq:change-of-coords}), for $p>2$ we have $F'\in L^{p}(\Omega)$ if and only if 
    \begin{align}\label{eq:main-series}
        \sum_{\gamma\in\Gamma_0} \frac{\left|\gamma'(0)\right|^{p}}{\abs{\rho(\gamma)'(0)}^{p-2}} <\infty.
    \end{align}
    In this section we describe how we estimate $\rho$ and $\Gamma_0$, and how we estimate the range of values of $p$ where the series (\ref{eq:main-series}) converges. \par Since $\Gamma$ is a free group of rank $2$ generated by $A$ and $B$ (see \S\ref{section:grafting}), it suffices to estimate $\rho^{-1}(A)$ and $\rho^{-1}(B)$ for generators $A$ and $B$ of $\Gamma$. As explained in \S\ref{section:grafting}, the M\"obius transformations $A$ and $B$ are given explicitly. To estimate $\rho^{-1}(A)=f^{-1}\circ A\circ f$ and $\rho^{-1}(B)=f^{-1}\circ B\circ f$, we estimate the Riemann mapping $f:\mathbb{D}\to\Omega$ using Driscoll's Schwarz-Christoffel toolbox \cite{driscoll-sc-0,driscoll-sc}, and then compute $f^{-1}\circ A\circ f$ and $f^{-1}\circ B\circ f$. 
    \par In \S\ref{subsec:polygonal} we describe how we get a polygon $\hat{\Omega}$ that is an approximation to $\Omega$. Then in \S\ref{subsec:est-conj-gen} we explain how we get estimates of $\rho^{-1}(A)$ and $\rho^{-1}(B)$ using this polygon and the Schwarz-Christoffel toolbox. Finally in \S\ref{subsec:vals} we describe how we show numerically that the sum $\sum_{\gamma\in\Gamma_0}\left|\gamma'(0)\right|^{p} \left|\rho(\gamma)'(0)\right|^{2-p}$ converges for $p=4$, and we check for which values of $p$ it diverges. 
    \subsection{Finite polygonal approximations}\label{subsec:polygonal}
    We first remind the reader of some notation from \S\ref{section:grafting}. Recall that we have a once-punctured torus group of M\"obius transformations $\bar{\Gamma}$ acting on $\mathbb{H}$, a map $\psi:\mathbb{H}\to\Omega$, and a group isomorphism $f:\bar{\Gamma}\to\Gamma$, so that $\psi$ is $f$-equivariant. Recall that $\eta=i\frac{1+z}{1-z}:\mathbb{H}\to\mathbb{D}$ is the M\"obius transformation we use to move from the upper half-plane model to the disk model.
    \par Since $\mathbb{D}/\bar{\Gamma}$ is a once-punctured torus, the limit set $\Lambda(\bar{\Gamma})$ of $\bar{\Gamma}$ is $\mathbb{S}^1$. Let $x_0=\eta(\coth\frac{\lambda}{4})=i\frac{1+\coth\frac{\lambda}{4}}{1-\coth\frac{\lambda}{4}}=-ie^{\lambda/2}$ be the $\eta$-image of an ideal vertex of the fundamental domain $\Phi$ for the action of $\bar{\Gamma}$ on $\mathbb{H}$. It is easy to see that $\bar{\Gamma}x_0$ is dense in $\mathbb{S}^1$. From the construction of $\psi$, we see that it extends continuously to a surjection $\psi:\mathbb{S}^1\to\partial\Omega$. Therefore $\psi(\bar{\Gamma}x_0)$ is dense in $\partial\Omega$ and by equivariance $\psi(\bar{\Gamma} x_0)=\Gamma\psi(x_0)=\Gamma x_0$.
    \par To obtain an approximation to $\partial\Omega$, we generate some number of random points $v_1,v_2,...,v_n$ of the form $\gamma x_0$ where $\gamma$ is a word in $A$ and $B$ of length at most some parameter $\ell>0$. These points determine the boundary of a polygon $\hat{\Omega}$ which is an approximation to $\Omega$.
    \par Our goal to use the Schwarz-Christoffel toolbox to compute the Riemann map of this polygon. For this we need the vertices $v_i$ to be sorted, which will almost never happen. To sort them, we note that it suffices to sort $\psi^{-1}(v_i)$. If $v_i=\gamma_ix_0$, then $\psi^{-1}(v_i)=f^{-1}(\gamma_i)\left(\coth\frac{\lambda}{4}\right)\in\mathbb{R}$ by equivariance.This formula is used to compute $\psi^{-1}(v_i)\in\mathbb{R}$.
    \par The content of this subsection is summarized in Procedure \ref{algo:polygon}. In this procedure, we use the following additional piece of notation. For a word $w\in \langle a,b\rangle$ in the free group $\langle a,b\rangle$ of rank $2$, and elements $x,y\in G$ in some group $G$, denote by $w(a,b)$ the element of $G$ obtained by substituting $x$ for $a$ and $y$ for $b$ in $w$.
    \begin{algorithm}[!ht]
        \small
        \DontPrintSemicolon
        \KwInput{The number of vertices $n$, and the maximum word length $\ell$.}
        \KwOutput{The vertices $v_1,v_2,...,v_n$ of a polygon.}
        \caption{Constructing finite polygonal approximations.}\label{algo:polygon}
        Construct a list $W_1,W_2,...,W_{2(3^\ell-1)}$ of all words of length at most $\ell$ in the free group $\langle a,b\rangle$ of rank $2$.\\
        Pick integers $1\leq i_1<i_2<...<i_n\leq 2(3^\ell-1)$ uniformly at random.\\
        Set $\gamma_k=W_{i_k}(A, B)$ and $\bar{\gamma}_k=W_{i_k}(\bar{A}, \bar{B})$.\\
        Reorder $i_1,i_2,...,i_n$ so that $\bar{\gamma}_1\left(\coth\frac{\lambda}{4}\right)<\bar{\gamma}_2\left(\coth\frac{\lambda}{4}\right)<...<\bar{\gamma}_n\left(\coth\frac{\lambda}{4}\right)$.\\
        Let $v_k=\gamma_k(-i\exp(\lambda/2))$.
    \end{algorithm}
    \subsection{Procedure for estimating conjugated generators}\label{subsec:est-conj-gen}
    Here we describe the method we use to estimate $\rho^{-1}(A)=f^{-1}\circ A\circ f$ and $\rho^{-1}(B)=f^{-1}\circ B\circ f$, assuming that we have an estimate of $f:\mathbb{D}\to\Omega$. We assume that we can compute $f(z)$ and $f^{-1}(w)$ for any $z\in\mathbb{D}, w\in\Omega$, which is the case for Riemann mappings computed using the Schwarz-Christoffel toolbox.
    \par More generally, we describe how to estimate $f^{-1}\circ X\circ f$ for any M\"obius transformation $X$ preserving $\Omega$. Since $X$ preserves $\Omega$, its conjugate $f^{-1}\circ X\circ f$ preserves the disk. It therefore takes the form 
    \begin{align*}
        f^{-1}\circ X\circ f=\lambda\frac{z-a}{1-\bar{a}z}.
    \end{align*}
    \par Our methods estimate $\lambda\in\mathbb{S}^1$ and $a\in\mathbb{D}$. They are summarized in Procedure \ref{algo:est-gen}. Note that we can compute $f$ and $f^{-1}$ at specific points. We generate some number $n$ of points $z_1,z_2,...,z_n\in\mathbb{D}$ and then compute $w_i=f^{-1}(X(f(z_i)))$. We estimate $\lambda,a$ so that $\lambda\frac{z_i-a}{1-\bar{a}z_i}$ is as close as possible to $w_i$. Specifically, we use the sum-of-squares error 
    \begin{align*}
        \sum_{i=1}^n \abs{w_i-\lambda\frac{z_i-a}{1-\bar{a}z_i}}^2.
    \end{align*}
    In our implementation, we use MATLAB's \texttt{lsqcurvefit} to do the optimization.
    \begin{algorithm}[!ht]
        \small
        \DontPrintSemicolon
        \KwInput{The Riemann mapping $f:\mathbb{D}\to\Omega$, a M\"obius transformation $X$ preserving $\Omega$, and the sample size $n$.}
        \KwOutput{Estimate $\hat{X}$ of the M\"obius transformation $f^{-1}\circ X\circ f$.}
        \caption{Estimating $f^{-1}\circ X\circ f$ given the Riemann mapping $f:\mathbb{D}\to\Omega$ and a M\"obius transformation $X$ preserving $\Omega$.}\label{algo:est-gen}
        Generate $n$ points $z_1,z_2,...,z_n\in\mathbb{D}$ uniformly at random.\\
        Compute $w_i=f^{-1}(X(f(z_i)))$.\\
        Look for the minimizers $(\hat{\lambda},\hat{a})$ of 
        \begin{align*}
            \sum_{i=1}^n \abs{w_i-\lambda \frac{z_i-a}{1-\bar{a}z_i}}^2
        \end{align*}
        over $(\lambda, a)\in\mathbb{S}^1\times\mathbb{D}$.\\
        Set $\hat{X}=\hat{\lambda}\frac{z-\hat{a}}{1-\overline{\hat{a}}z}$.
    \end{algorithm}
    \subsection{Estimates of the conjugated generators}
    To actually estimate generators of $\Gamma_0=f^{-1}\circ\Gamma\circ f$, we combine the procedures described in \S\ref{subsec:polygonal} and \S\ref{subsec:est-conj-gen}.
    \par We first compute polygonal approximations to $\Omega$ using Procedure \ref{algo:polygon}. We set the maximum word length $\ell=12$. We construct the following approximations:
    \begin{enumerate}
        \item the main polygon, $\hat{\Omega}$ that has $n=600$ vertices, and
        \item several smaller polygons, $\hat{\Omega}^\text{s}_i$ for $1\leq i\leq 20$, each with $n=100$ vertices.
    \end{enumerate}
    We will use $\hat{\Omega}$ to obtain estimates of $f^{-1}\circ A\circ f$ and $f^{-1}\circ B\circ f$, and the smaller polygons $\hat{\Omega}^\text{s}_i$ to validate these estimates.
    \par Using the Schwarz-Christoffel toolbox, we compute the Riemann mappings $f:\mathbb{D}\to\hat{\Omega}$ and $f_i^\text{s}:\mathbb{D}\to\hat{\Omega}^\text{s}_i$. We use Procedure \ref{algo:est-gen} with $f$, and $X=A$ and $X=B$. We denote the resulting estimates $\lambda^A, a^A$ and $\lambda^B, a^B$, and call them the main estimates in this subsection. These give the estimates of $f^{-1}\circ A\circ f$ and $f^{-1}\circ B\circ f$ that we will use in the rest of the paper.
    \par For validation, we also run Procedure \ref{algo:est-gen} with $\hat{\Omega}^\text{s}_i$ for $i=1,2,...,20$, and $X=A$ and $X=B$. We call the resulting estimates $\lambda_i^A, a_i^A$ and $\lambda_i^B, a_i^B$. We plot these and the main estimates in Figures \ref{verify-a} and \ref{verify-b}. We see that for each parameter, there is a clear cluster in the estimates obtained from $\hat{\Omega}_i^\text{s}, i=1,2,...,20$ centered approximately at the estimate obtained from $\hat{\Omega}$. This suggests that $\lambda^A,a^A$ and $\lambda^B, a^B$ are accurate. 
    \begin{dfn}
        We set $\hat{\Gamma}_0=\langle \hat{A}, \hat{B}\rangle$, where 
        \begin{align*}
            \hat{A}(z)=\lambda^A \frac{z-a^A}{1-\overline{a^A}z},
        \end{align*}
        and 
        \begin{align*}
            \hat{B}(z)=\lambda^B\frac{z-a^B}{1-\overline{a^B}z}.
        \end{align*}
        We also define the homomorphism $\hat{\rho}:\hat{\Gamma}_0\to\Gamma$ by setting $\hat{\rho}(\hat{A})=A$ and $\hat{\rho}(\hat{B})=B$.
    \end{dfn}
    \begin{figure}
        \includegraphics[width=\textwidth, keepaspectratio]{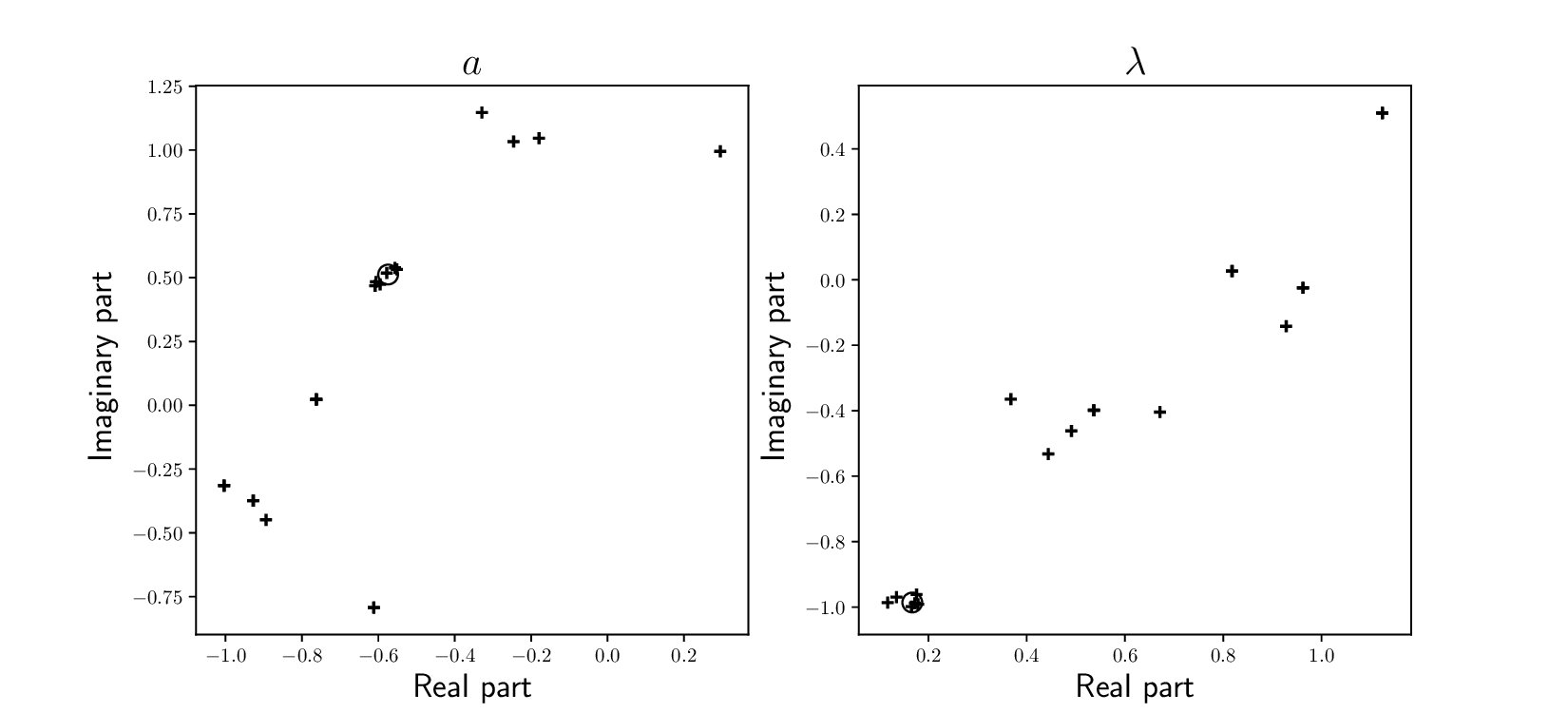}
        \caption{The crosses represent values of parameters $\lambda^A_i, a^A_i$ of a disk automorphism estimated using smaller $100$-point polygonal approximations $\hat{\Omega}_i^\text{s}$ to $\Omega$. The circles are the estimates of $\lambda^A$ and $a^A$ obtained from $\hat{\Omega}$. }\label{verify-a}
    \end{figure}
    \begin{figure}
        \includegraphics[width = \textwidth, keepaspectratio]{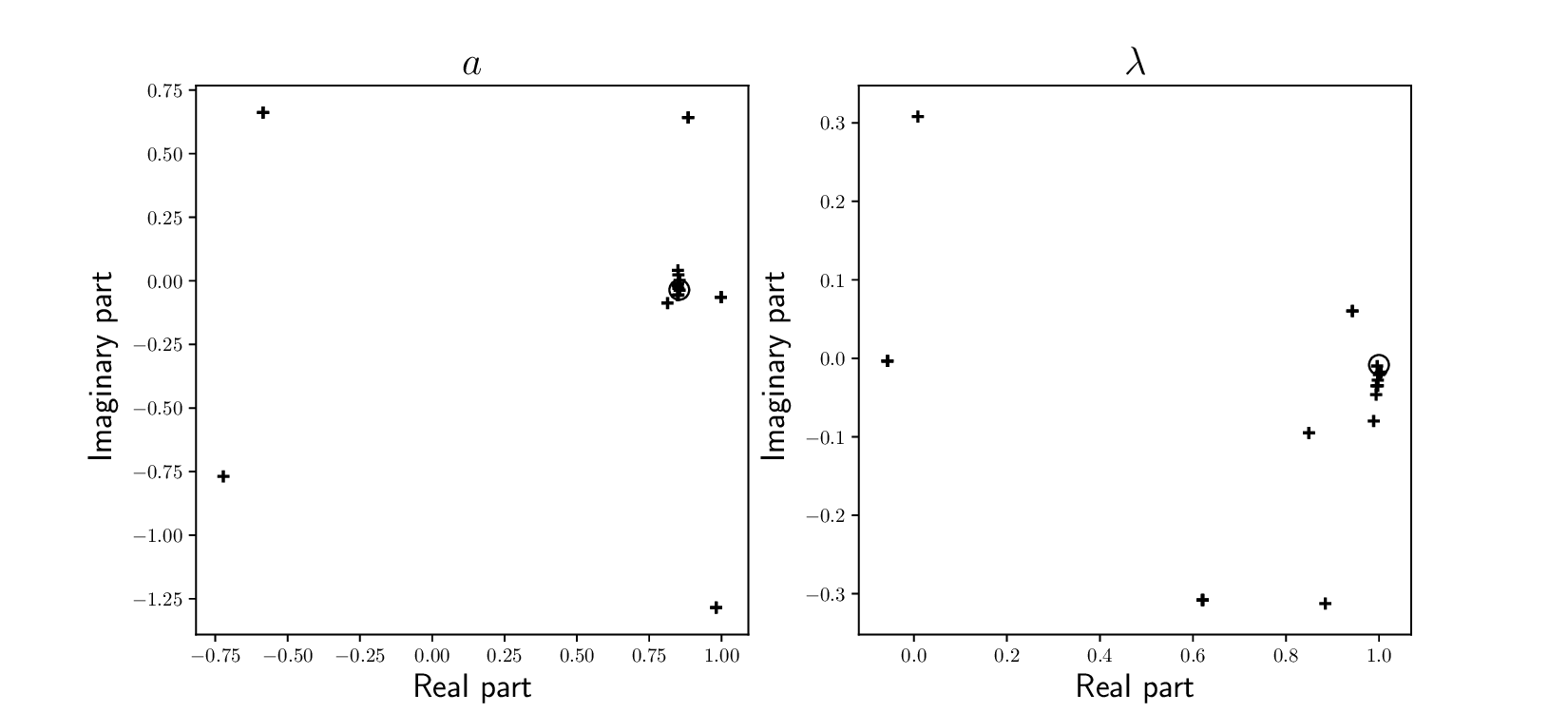}
        \caption{The crosses represent values of parameters $\lambda^B_i, a^B_i$ of a disk automorphism estimated using smaller $100$-point polygonal approximations $\hat{\Omega}_i^\text{s}$ to $\Omega$. The circles are the estimates of $\lambda^B$ and $a^B$ obtained from $\hat{\Omega}$. }\label{verify-b}
    \end{figure} 
    \subsection{Values of $p$ for which $F'\in L^p(\Omega)$} 
    \label{subsec:vals}
    By Theorem \ref{thm:main-estimate} and (\ref{eq:change-of-coords}), $F'\in L^p(\Omega)$ if and only if
    \begin{align*}
        \sum_{\gamma\in\Gamma_0} \left|\gamma'(0)\right|^{p}\left| \rho(\gamma)'(0)\right|^{2-p}<\infty.
    \end{align*} We are hence interested in values of $p$ for which 
    \begin{align*}
        \sum_{\gamma\in \hat{\Gamma}_0} \left\lvert \gamma'(0)\right\rvert^p \left\lvert \hat{\rho}(\gamma)'(0)\right\rvert^{2-p}<\infty.
    \end{align*}
    We define $\hat{\Gamma}_0^n$ to be the set of all elements that have word length $n$ in the generators $\hat{A},\hat{B}$ of $\hat{\Gamma}_0$. We then split the above sum into 
    \begin{align*}
        \sum_{n\geq 0} S_n^{(p)}\text{ where }S_n^{(p)}=\sum_{\gamma\in\hat{\Gamma}_0^n} \left|\gamma'(0)\right|^p\left|\hat{\rho}(\gamma)'(0)\right|^{2-p}.
    \end{align*}
    We numerically confirm that $S_n^{(4)}$ decays exponentially with $n$, see Figure \ref{fig:result_4}.
    \begin{figure}
        \includegraphics[width = \textwidth, keepaspectratio]{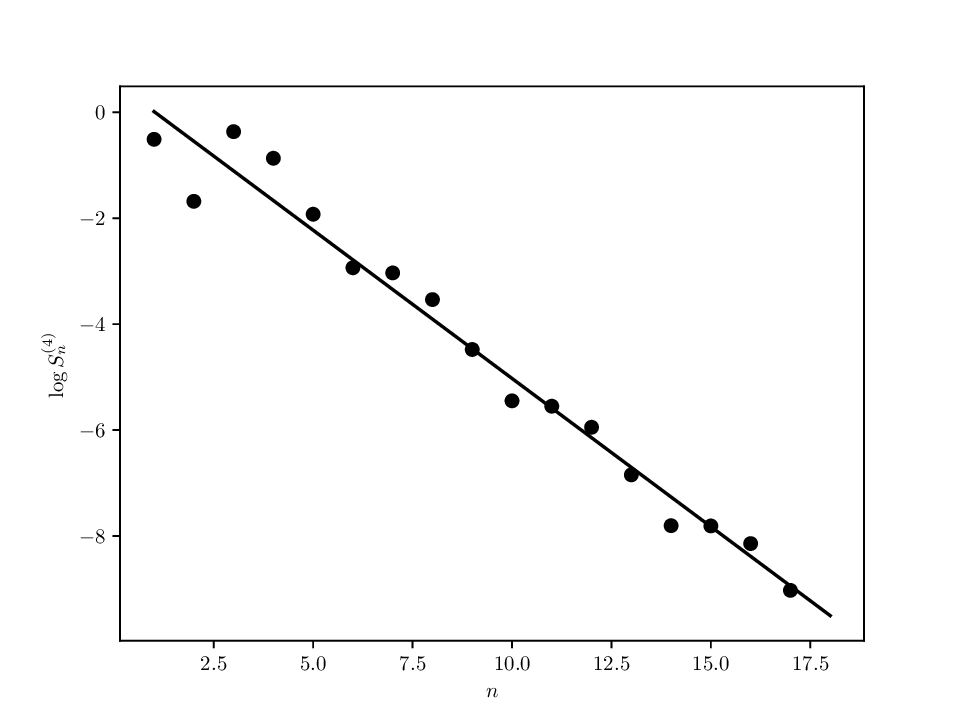}
        \caption{The plot of $\log S_n^{(4)}$ as a function of $n$ for $n\leq 17$. The solid line is the least-squares best linear fit to the shown data points. }\label{fig:result_4}
    \end{figure}
    \par Recall the definition of $p_\star$ from the Main result in \S\ref{subsec:results}, $p_\star=\sup\{p>0:F'\in L^p(\Omega)\}$. By Theorem \ref{thm:main-estimate} and (\ref{eq:change-of-coords}), this is equal to 
    \begin{align*}
        p_\star&=\sup\left\{p>0:\sum_{n\geq 0}S_n^{(p)}<\infty\right\}.
    \end{align*}
    How close $p_\star$ is to $4$ gives an indication of how close the domain $\Omega$ is to being a counterexample to Brennan's conjecture. Numerical results show that $S_n^{(5.52)}$  decays exponentially with $n$ and that $S_n^{(5.54)}$ grows exponentially with $n$, see Figure \ref{fig:result_bound}, meaning that $5.52<p_\star<5.54$, showing the Main result. \begin{figure}
        \includegraphics[width=\textwidth, keepaspectratio]{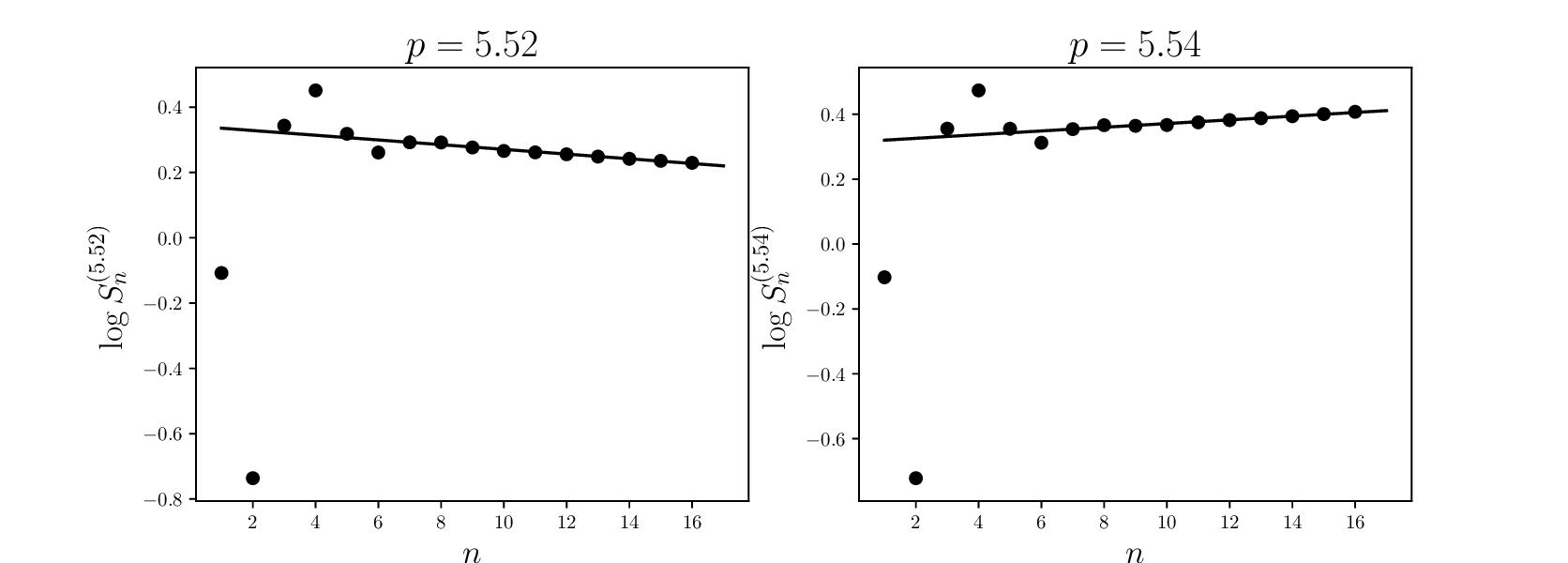}
        \caption{The plots of $\log S_n^{(p)}$ as a function of $n$ for $n\leq 16$ and $p=5.52, 5.54$. The solid lines are the least-squares best linear fits to the data points $(n, \log S_n^{(p)})$ for $n\geq 6$, for $p=5.52, 5.54$. }\label{fig:result_bound}
    \end{figure}
    \FloatBarrier
    
\end{document}